\documentclass[12pt,a4paper]{article}

\usepackage{amsmath,amsthm,amssymb}
\usepackage{a4wide}
\usepackage{hyperref,url}
\usepackage{tikz}

\newtheorem{theorem}{Theorem}[section]
\newtheorem{lemma}[theorem]{Lemma}
\newtheorem{corollary}[theorem]{Corollary}
\newtheorem{conjecture}[theorem]{Conjecture}

\newcommand{\NOS}{\mathbb{N}}
\newcommand{\OS}{\mathbb{S}}

\tikzstyle{sommet}=[circle,draw,thick,fill=white]
\tikzstyle{sommetn}=[circle,fill=black,scale=0.6]
\tikzstyle{sommetp}=[circle,draw,thick,fill=white,scale=0.6]
\tikzstyle{point}=[dashed]
\tikzstyle{gras}=[ultra thick]

\def\longbox#1{\parbox{0.85\textwidth}{#1}}

\title{Long induced paths in graphs\thanks{This work was partially
supported by ANR Project Stint (\textsc{anr-13-bs02-0007}), and LabEx
PERSYVAL-Lab (\textsc{anr-11-labx-0025}).}}

\author{Louis~Esperet\thanks{CNRS, Laboratoire G-SCOP, Universit\'e de
Grenoble-Alpes, France.}
\and
Laetitia~Lemoine\thanks{Laboratoire G-SCOP, Universit\'e de
Grenoble-Alpes, France.}
\and
Fr\'ed\'eric Maffray\thanks{CNRS, Laboratoire G-SCOP, Universit\'e de
Grenoble-Alpes, France.}}

\begin{document}
\maketitle

\begin{abstract}
We prove that every 3-connected planar graph on $n$ vertices contains
an induced path on $\Omega(\log n)$ vertices, which is best possible
and improves the best known lower bound by a multiplicative factor of
$\log \log n$.  We deduce that any planar graph (or more generally,
any graph embeddable on a fixed surface) with a path on $n$ vertices,
also contains an induced path on $\Omega(\sqrt{\log n})$ vertices.  We
conjecture that for any $k$, there is a positive constant $c(k)$ such
that any $k$-degenerate graph with a path on $n$ vertices also
contains an induced path on $\Omega((\log n)^{c(k)})$ vertices.  We
provide examples showing that this order of magnitude would be best
possible (already for chordal graphs), and prove the conjecture in the
case of interval graphs.
\end{abstract}

\section{Introduction}

A graph contains a long induced path (i.e., a long path as an induced
subgraph) only if it contains a long path.  However, this necessary
condition is not sufficient, as shown by complete graphs and complete
bipartite graphs.  On the other hand, it was proved by Atminas, Lozin
and Ragzon \cite{ALR} that if a graph $G$ contains a long path, but
does not contain a large complete graph or complete bipartite graph,
then $G$ contains a long induced path.  Their proof uses several
applications of Ramsey theory, and the resulting bound on the size
of a long induced path is thus quantitatively
weak.

The specific case of $k$-degenerate graphs (graphs such that any
subgraph contains a vertex of degree at most $k$) was considered by
Ne\v{s}et\v{r}il and Ossona de Mendez in \cite{NO}.  These graphs
clearly satisfy the assumption of the result of Atminas, Lozin and
Ragzon \cite{ALR}, so $k$-degenerate graphs with long paths also
contain long induced paths.  Ne\v{s}et\v{r}il and Ossona de Mendez
\cite[Lemma 6.4]{NO} gave the following more precise bound: if $G$ is
$k$-degenerate and contains a path of size $n$, then it contains an
induced path of size $\frac{\log \log n}{\log (k+1)}$.  This result
was then used to characterize the classes of graphs of bounded
tree-length precisely as the classes of degenerate graphs excluding
some induced path of fixed size.  Ne\v{s}et\v{r}il and Ossona de
Mendez also asked \cite[Problem 6.1]{NO} whether their doubly
logarithmic bound could be improved.

Arocha and Valencia \cite{AV} considered the case of $3$-connected
planar graphs and $2$-connected outerplanar graphs.  An
\emph{outerplanar graph} is a graph that can be drawn in the the plane
without crossing edges and with all vertices on the external face.  It
was proved in \cite{AV} that in any $2$-connected outerplanar graph
with $n$ vertices, there is an induced path with $\Omega(\sqrt{\log
n})$ vertices and, using this fact, that any $3$-connected planar
graph with $n$ vertices contains an induced path with
$\Omega(\sqrt[3]{\log n})$ vertices.  Note that in these results there
is no initial condition on the size of a long path (the bounds only
depend on the number of vertices in the graph).  Di Giacomo, Liotta
and Mchedlidze \cite{GLM} recently proved that any $n$-vertex
$3$-connected planar graph contains an induced outerplanar graph of
size $\sqrt[3]{n}$, and that any $n$-vertex $2$-connected outerplanar
graph contains an induced path of size $\frac{\log n}{2 \log \log n}$,
and combining these two bounds, that any $n$-vertex $3$-connected
planar graph contains an induced path of size $\frac{\log n}{12 \log
\log n}$.

We will prove that if a $k$-tree (defined in the next section)
contains a path of size $n$, then it contains an induced path of size
$\tfrac{\log n}{k \log k}$.  Using similar ideas, we will show that a
partial 2-tree with a path of size $n$ also contains an induced path
of size $\Omega(\log n)$.  Outerplanar graphs are partial 2-trees, and
2-connected outerplanar graphs are Hamiltonian, so in particular this
shows that any $n$-vertex $2$-connected outerplanar graph contains an
induced path of size $\Omega(\log n)$.  Using the results of Di
Giacomo, Liotta and Mchedlidze \cite{GLM}, this directly implies that
any $n$-vertex $3$-connected planar graph contains an induced path of
size $\Omega(\log n)$, improving their bound by a multiplicative
factor of $\log \log n$.  Our bounds are tight up to a constant
multiplicative factor.

We derive from our result on $3$-connected planar graphs that any
planar graph (and more generally, any graph embeddable on a fixed
surface) with a path on $n$ vertices contains an induced path of
length $\Omega(\sqrt{\log n})$.  We also construct examples of planar
graphs with paths on $n$ vertices in which all induced paths have size
$O(\tfrac{\log n}{\log \log n})$.  Our examples can be seen as special
cases of a more general family of graphs: chordal graphs with maximum
clique size $k$, containing a path on $n$ vertices, but in which every
induced path has size $O((\log n)^{\tfrac2{k-1}})$.  This shows that
the doubly logarithmic bound of Ne\v{s}et\v{r}il and Ossona de Mendez
\cite{NO} cannot be replaced by anything better than $(\log
n)^{c(k)}$, for some function $c$.  We believe that this is the
correct order of magnitude.

\begin{conjecture}\label{conj:1}
There is a function $c$ such that for any integer $k$, any
$k$-degenerate graph that contains a path of size $n$ also contains an
induced path of size $(\log n)^{c(k)}$.
\end{conjecture}

We prove this conjecture in the special case of interval graphs.  More
precisely, we show that any interval graph with maximum clique size
$k$ containing a path of size $n$ contains an induced path of size
$\Omega((\log n)^{\frac{1}{(k-1)^2}})$, where the hidden multiplicative
constant depends on $k$.

%%%%%%%

We finish this section by recalling some definitions and terminology.
In a graph $G$, we say that a vertex $x$ is \emph{complete} to a set
$S\subseteq V(G)\setminus x$ when $x$ is adjacent to every vertex in
$S$.  A \emph{block} of a graph $G$ is any maximal $2$-connected
induced subgraph of $G$, where a bridge (a cut-edge) is also a block.
It is well known that the intersection graph of the blocks and
cut-vertices of $G$ can be represented by a tree $T$, which we call
the \emph{block tree} of $G$.  The \emph{size} of a path is the number
of its vertices, and the \emph{length} of a path is the number of its
edges.  A vertex is \emph{simplicial} if its neighborhood is a clique.
A simplicial vertex of degree $k$ is called \emph{$k$-simplicial}.  A
graph is \emph{chordal} if it contains no induced cycle of length at
least four.

In this article, the base of the logarithm is always assumed to be 2.
%%%%%%%

\section{Induced paths in $k$-trees}

For any integer $k\ge 1$, the class of $k$-trees is defined
recursively as follows:
\begin{quote}
$\bullet$ Any clique on $k$ vertices is a $k$-tree.

$\bullet$ If $G$ has a $k$-simplicial vertex $v$, and $G\setminus v$
is a $k$-tree, then $G$ is a $k$-tree.
\end{quote}
Hence if $G$ is any $k$-tree on $p$ vertices, there is an ordering
$x_1, \ldots, x_p$ of its vertices such that $\{x_1, \ldots, x_k\}$
induces a clique and, for each $i=k+1, \ldots, p$, the vertex $x_i$ is
a $k$-simplicial vertex in the subgraph induced by $\{x_1, \ldots,
x_i\}$.  We call this a \emph{$k$-simplicial ordering}, and we call
$\{x_1, \ldots, x_k\}$ the \emph{basis} of this ordering.  We recall
some easy properties of $k$-trees.
\begin{lemma}\label{lem:ksimp}
Let $G$ be any $k$-tree.  Then $G$ is chordal.  Moreover, $G$
satisfies the following properties:
\begin{itemize}
\item[(i)]
If $G$ is not a $k$-clique, then every maximal clique in $G$ has size
$k+1$, and $G$ has exactly $|V(G)|-k$ maximal cliques.
\item[(ii)]
If $G$ is not a clique, then $G$ has two non-adjacent $k$-simplicial
vertices.
\item[(iii)]
Any $k$-clique can be taken as the basis of a $k$-simplicial ordering
of $G$.
\item[(iv)]
For any $k$-clique $K$ of $G$, every component of $G\setminus K$
contains exactly one vertex that is complete to $K$.
\end{itemize}
\end{lemma}
\begin{proof}
Properties (i) and (ii) follow easily from the existence of a
$k$-simplicial ordering, and we omit the details.

We prove (iii) by induction on $p=|V(G)|$.  Consider any $k$-clique
$K$ of $G$.  If $G=K$, there is nothing to prove.  So assume that $G$
is not a $k$-clique.  By (ii), $G$ has two non-adjacent $k$-simplicial
vertices $x$ and $y$.  We may assume that $x\notin K$.  By the
induction hypothesis, $G\setminus x$ admits a $k$-simplicial ordering
$x_1, \ldots, x_{p-1}$ such that $K$ is the basis of this ordering.
Then $x_1, \ldots, x_{p-1},x$ is a $k$-simplicial ordering for $G$,
with $K$ as a basis.

To prove (iv), consider any component $A$ of $G\setminus K$.  By
(iii), there is a $k$-simplicial ordering with $K$ as a basis.  The
first vertex of $A$ in the ordering has no neighbor in $V(G)\setminus
(K\cup A)$, so it must be complete to $K$.  Now suppose that $A$
contains two vertices $x,y$ that are complete to $K$.  Let
$x_0$-$\cdots$-$x_q$ be a shortest path in $A$ with $x=x_0$ and
$y=x_q$.  The vertex $x_1$ has a non-neighbor $z\in K$, for otherwise
$K\cup\{x_0, x_1\}$ is a clique of size $k+2$, contradicting (i).  Let
$j\ge 2$ be the smallest integer such that $x_j$ is adjacent to $z$;
so $2\le j\le q$.  Then $\{x_0, x_1, \ldots, x_j, z\}$ induces a cycle
of length at least four in $G$, contradicting the fact that $G$ is
chordal.
\end{proof}

\begin{theorem}\label{ktree}
Let $k$ be a fixed integer, $k\ge 2$, and let $G$ be a $k$-tree that
contains an $n$-vertex path.  Then $G$ contains an induced path of
size $\frac{\log (n-k-1)}{k \log k}=\frac{\log n}{k \log
k}-O(\frac{1}{n})$.
\end{theorem}

%%%preuve dans les k-tree%%%
\begin{proof}
Let $P$ be a path on $n$ vertices in $G$.  We may assume that $G$ is
minimal with these properties; in other words, if $G$ has a vertex $x$
such that $G\setminus x$ is a $k$-tree and contains $P$, then it
suffices to prove the theorem for $G\setminus x$; so we may assume
that there is no such vertex.  We claim that:
\begin{equation}\label{kp1c}
\longbox{If $K$ is any $k$-clique in $G$, then ${G}\setminus K$
contains at most $k+1$ vertices that are complete to $K$.}
\end{equation}
Proof of (\ref{kp1c}): Whenever $P$ goes from one component of
$G\setminus K$ to another component, it must go through at least one
vertex of $K$.  This implies that $P$ goes through at most $k+1$
components of $G\setminus K$.  On the other hand, $P$ must go through
each component $A$ of $G\setminus K$, for otherwise we can restrict
ourselves to $G\setminus A$, which is a $k$-tree since we can take $K$
as a basis of a $k$-simplicial ordering, and this contradicts the
minimality of $G$.  Hence $G\setminus K$ has at most $k+1$ components;
and by Lemma~\ref{lem:ksimp}~(iv), we deduce that (\ref{kp1c}) holds.

\medskip

Let $K_0$ be a fixed $k$-clique in $G$.  We associate with $G$ a
labelled rooted tree $T(G)$, where each node $v$ has a label $L(v)$
which is a set of $k$-cliques of $G$, with the following properties:
\begin{itemize}
\item
The root $r$ of $T(G)$ is a new node, and its label is
$\{K_0\}$;
\item
$V(T(G))\setminus\{r\} = V(G)\setminus K_0$;
\item
The label of any non-root node $v$ consists of $k$ $k$-cliques of $G$
that contain $v$;
\item
Every $k$-clique of $G$ is in the label of exactly one node of $T(G)$;
\item
For any two nodes $u,v$ in $T(G)$ such that $v$ is a child of $u$,
there is a $k$-clique $K\in L(u)$ such that $v$ is complete to $K$ in
$G$ and
each element of $L(v)$ is a subset of $K\cup \{v\}$.
\end{itemize}
We prove the existence of such a tree $T(G)$ by induction as follows.
If $G$ is a $k$-clique, then $K_0=V(G)$, and $T(G)$ is the tree with a
unique node $r$, whose label is $\{V(G)\}$.  Now suppose that
$G$ is not a $k$-clique.  By Lemma~\ref{lem:ksimp}~(iii), $G$ has a
$k$-simplicial vertex $x$ with $x\notin K_0$.  Let $K$ be the
neighborhood of $x$, so $K$ is a $k$-clique of $G\setminus x$.  By the
induction hypothesis, $G\setminus x$ admits a labelled rooted tree
$T(G\setminus x)$ that satisfies the properties above.  Let $u$ be the
unique node of $T(G\setminus x)$ whose label contains $K$.  Then
$T(G)$ is obtained from $T(G\setminus x)$ by adding $x$ as a child of
$u$, and we set $L(x)= \{(K\cup\{x\})\setminus y \mid$ for all $y\in
K\}$.  Hence the $k$-cliques that contain $x$ are in $L(x)$.  It
follows that every $k$-clique of $G$ is in the label of exactly one
node of $T(G)$.  So all the required properties hold for $T(G)$.

% \begin{figure}
% \begin{center}
% \begin{tikzpicture}[scale=0.7]
% \begin{scope}
% \node[sommetp] (1) at (0,0) {1};
% \node[sommetp] (2) at (2,0) {2};
% \node[sommetp] (3) at (1,2) {3};
% \node[sommetp] (4) at (3,2) {4};
% \node[sommetp] (5) at (2,1.5) {5};
% \draw (1)-- (2);
% \draw (2) -- (3) --(1);
% \draw (2) -- (4) -- (3);
% \draw (2) -- (5) -- (3);
% \end{scope}
% \begin{scope}[xshift=200]
% \node[sommetn] (1) at (0,2) {};
% \node (1') at (1,2) {$\{12\}$};
% \node[sommetn] (2) at (0,1) {};
% \node (2') at (1.8,1) {$\{13,23\}$};
% \node[sommetn] (3) at (0,0) {};
% \node (3') at (-1.5,0) {$\{24, 34\}$};
% \node[sommetn] (4) at (2,0) {};
% \node (4') at (3.5,0) {$\{25,35\}$};
% \draw (1) -- (2);
% \draw (2) -- (3);
% \draw (2) -- (4);
% \end{scope}
% \end{tikzpicture}
% \end{center}
%    \caption{A $2$-tree $G$ and the labelled tree $T(G)$}
% \end{figure}

\begin{figure}[htbp]
\begin{center}
\includegraphics[scale=1.2]{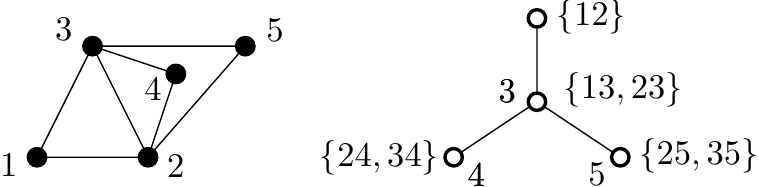}
\caption{A $2$-tree $G$ and a corresponding labelled tree $T(G)$. \label{fig:tree}}
\end{center}
\end{figure}

%on trouve un grand chemin dans l'arbre%
Now we claim that:
\begin{equation}\label{degtg}
\mbox{Each node of $T({G})$ has degree at most $k^2+1$.}
\end{equation}
Let $u$ be any node of $T({G})$ and $v$ be any child of $u$.  By the
properties of $T({G})$, the vertex $v$ is complete to a member $K$ of
$L(u)$.  By (\ref{kp1c}), at most $k+1$ such vertices exist for each
$K$.  Hence if $u$ is the root, its degree is at most $k+1$.  Now
suppose that $u$ is not the root, and let $t$ be its parent node in
$T(G)$.  By the properties of $T(G)$, $u$ is complete (in $G$) to some
$k$-clique $Q\in L(t)$, and we have $K\subset Q\cup \{u\}$. Note that
the unique vertex in $Q\cup \{u\} \setminus K$ is complete to $K$ and
is either an ancestor of $u$
in $T(G)$, or a vertex of $K_0$, and therefore does not appear among the children of $u$. So $u$ has at most $k(k+1-1)=k^2$ children; hence the degree
of $u$ in $T(G)$ is at most $k^2+1$.  Thus (\ref{degtg}) holds.

\medskip

Let $\ell$ be the size of a longest path in $T({G})$, and let $P=
v_1$-$\cdots$-$v_{\ell}$ be a path in $T({G})$.  We claim that:
\begin{equation}\label{tpath}
\ell\ge\frac{\log (n-k-1)}{\log k}.
\end{equation}
Since $|V(G)|\ge n$, it follows that $T({G})$ has at least $n-k+1$
nodes.  First suppose that $\ell$ is odd.  Let $m=(\ell+1)/2$.  So the
vertex $v_m$ is the middle vertex of $P$, and every vertex of $T(G)$
is at distance at most $m-1$ from $v_m$.  It follows that $n-k+1\le 1+
(k^2+1) (k^2)^{m-2}$, so $n-k\le k^{2m-1}$, whence $\ell \ge
\frac{\log(n-k)}{\log k}$.  \\
Now suppose that $\ell$ is even.  Let $m=\ell/2$.  So the edge
$v_mv_{m+1}$ is the middle edge of $P$, and every vertex of $T(G)$ is
at distance at most $m-1$ from one of $v_m, v_{m+1}$.  It follows that
$n-k+1\le 2+ (k^2)^{m-1}$, so $n-k-1\le k^{2m-2}$, whence $\ell \ge
\frac{\log(n-k-1)}{\log k} +2$.  Thus (\ref{tpath}) holds.

\medskip

Using the path $P$ we construct, by induction on $i=1, \ldots, \ell$,
a collection of $k$ vertex-disjoint induced paths $P_1, \dots, P_k$ in
${G}$, adding one vertex of $G$ at each step, so that the following
properties hold at each step~$i$, where $u_{j,i}$ is the last vertex
of $P_j$:
\begin{quote}
$\bullet$ The set $K_i=\{u_{1,i}, \ldots, u_{k,i}\}$ is a $k$-clique
of ${G}$, with $K_i\in L(v_i)$; \\
$\bullet$ If $i\ge 2$ then $|K_{i-1}\cap K_i|= k-1$.
\end{quote}
We do this as follows.  First pick one member $K_1$ of $L(v_1)$, and
for each $j=1,\ldots, k$ let the first vertex of $P_j$ be the $j$-th
vertex of $K_1$.

At step $i+1$, we consider two cases since $v_i$ is either a child of
$v_{i+1}$ or its parent.

Suppose that $v_i$ is a child of $v_{i+1}$.  Let $x_i$ be the vertex
of $G$ in $L(v_i)$, and let $K_{i+1}$ be the $k$-clique in
$L(v_{i+1})$ such that $x_i$ is complete to $K_{i+1}$.  Then
$|K_{i+1}\cap K_i|= k-1$.  Let $j$ be the unique integer such that
$u_{j,i} \in K_i\setminus K_{i+1}$, and consider the unique vertex $y
\in K_{i+1}\setminus K_i$.  Then, we take $u_{j,i+1}=y$ and
$u_{a,i+1}=u_{a,i}$ for all $a \neq j$.  In this case the vertices
$v_1, \dots, v_{i-1}$ are all descendants of $v_i$ in $T({G})$, and
have been added in the construction of $T({G})$ as descendants of the
clique $K_i$.  Since $u_{j,i+1}=y \not\in K_i$, among vertices of
$P_1, \dots, P_k$ the vertex $u_{j,i+1}$ is adjacent only to $u_{1,i},
\dots, u_{k,i}$; so the paths $P_1, \dots, P_k$ remain induced.

Now suppose that $v_i$ is the parent of $v_{i+1}$.  First suppose that
$i+1 \neq \ell$.  Then $v_{i+2}$ is a child of $v_{i+1}$.  Let
$x_{i+2}$ be the vertex in $L(v_{i+2})$, and let $K_{i+1}$ be the
$k$-clique in $L(v_{i+1})$ such that $x_{i+2}$ is complete to
$K_{i+1}$.  Then $|K_{i+1}\cap K_i|=k-1$.  Let $j$ be the unique
integer such that $u_{j,i} \in K_i\setminus K_{i+1}$, and let $y$ be
the unique vertex in $K_{i+1}\setminus K_i$.  Then we take
$u_{j,i+1}=y$ and $u_{a,i+1}=u_{a,i}$ for all $a \neq j$.  Since the
only neighbors of $y$ are the vertices in $K_i$ and vertices
corresponding to descendants of $v_{i+1}$ (which are not vertices in
the paths $P_1, \dots, P_k$), the paths $P_1, \dots, P_k$ remain
induced.

Finally suppose that $i+1 = \ell$.  Let $x_\ell$ be the vertex of $G$
that belongs to $L(v_\ell)$.  Then we can add $x_\ell$ to any of the
paths, say to $P_1$.  Since $x_\ell$ is adjacent to the vertices
$u_{1,i}, \dots, u_{k,i}$ only, the paths $P_1, \dots, P_k$ remain
induced.  This completes the construction of these paths.

Since $P$ has size $\ell$, there are $k+ \ell-1$ vertices in $P_1 \cup
\dots \cup P_k$.  These paths are disjoint, so one of them has size
at least $\frac{\ell+k-1}{k}\ge \frac{\ell}{k}$.

In summary, if $G$ contains a path of size $n$, then it contains an
induced path of size $\frac{\log (n-k-1)}{k \log k}$.  This completes the proof of the theorem.
\end{proof}

This bound is optimal, up to a constant multiplicative factor of $2k
\log k$.  To see this, consider the family of graphs $G_i$ depicted in
Figure~\ref{outerplanar}.  These examples were found by Arocha and
Valencia~\cite{AV}.  The graph $G_0$ is a triangle, and $G_i$ is
obtained from $G_{i-1}$ by adding, for each edge $uv$ created at step
$i-1$, a new vertex adjacent to $u$ and $v$.  Clearly these graphs are
outerplanar and 2-trees.  Moreover they are Hamiltonian.  The graph $G_i$ has
$n=3\times2^{i}$ vertices and therefore contains a path with the same
number of vertices, while it is easy to check that the longest induced
path in $G_i$ has size $2(i+1)=2 \log n+(2-2 \log 3)$.

Now, add $k-2$ universal vertices to each $G_i$. We obtain again
Hamiltonian $k$-trees with $n$ vertices in which all induced paths have
size at most $2 \log n$, as desired. It would be interesting to construct
examples such the size of the longest induced paths decreases as $k$
grows (for instance of order $\tfrac{\log n}{\log k}$). We have not
been able to do so.

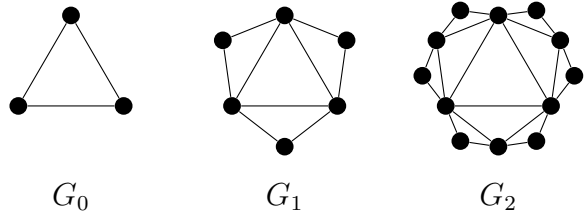
\begin{figure}[h]
   \begin{center}
   \begin{tikzpicture}[scale=0.2]
     \begin{scope}
\node[sommetn] (1) at (90:4) {};
\node[sommetn] (2) at (210:4) {};
\node[sommetn] (3) at (330:4) {};
\draw (1) -- (2) -- (3) -- (1);
\node at (270:8) {$G_0$};
\end{scope}
   \begin{scope}[xshift=400]
\node[sommetn] (1) at (90:4) {};
\node[sommetn] (2) at (210:4) {};
\node[sommetn] (3) at (330:4) {};
\node[sommetn] (4) at (150:4.7) {};
\node[sommetn] (5) at (270:4.7) {};
\node[sommetn] (6) at (30:4.7) {};
\draw (1) -- (2) -- (3) -- (1);
\draw (1) -- (4) -- (2) -- (5) -- (3) -- (6) -- (1);
\node at (270:8) {$G_1$};
\end{scope}
   \begin{scope}[xshift=800]
\node[sommetn] (1) at (90:4) {};
\node[sommetn] (2) at (210:4) {};
\node[sommetn] (3) at (330:4) {};
\node[sommetn] (4) at (150:4.7) {};
\node[sommetn] (5) at (270:4.7) {};
\node[sommetn] (6) at (30:4.7) {};
\node[sommetn] (7) at (120:5) {};
\node[sommetn] (8) at (180:5) {};
\node[sommetn] (9) at (240:5) {};
\node[sommetn] (10) at (300:5) {};
\node[sommetn] (11) at (0:5) {};
\node[sommetn] (12) at (60:5) {};
\draw (1) -- (2) -- (3) -- (1);
\draw (1) -- (4) -- (2) -- (5) -- (3) -- (6) -- (1);
\draw (1) -- (7) -- (4) -- (8) -- (2) -- (9) -- (5) -- (10) -- (3) -- 
(11) -- (6) -- (12) -- (1);

\node at (270:8) {$G_2$};
\end{scope}
   \end{tikzpicture}
   \end{center}
   \caption{A family of outerplanar graphs}
   \label{outerplanar}
\end{figure}

\subsection{Partial $2$-trees}

A \emph{partial $k$-tree} is any subgraph of a $k$-tree.  The
\emph{tree-width} of a graph $G$ is the least $k$ such that $G$ is a
partial $k$-tree.  Note that Theorem~\ref{ktree} has no direct
corollary on the size of long induced paths in partial $k$-trees in
general, but we can still deduce an asymptotically optimal bound for
the class of partial $2$-trees.  Before doing so, we prove the
following lemma which will be useful in several proofs.  Recall that a
class of graphs is called \emph{hereditary} if it is closed under
taking induced subgraphs.

\begin{lemma}\label{2co}
Let $\mathcal{F}$ be a hereditary family of graphs.  Suppose that
there are reals $\alpha,\beta>0$ such that every $2$-connected graph
in $\mathcal{F}$ that contains an $n$-vertex path contains an induced
path of size $\alpha (\log n)^\beta$.  Let $G$ be any connected graph
in $\mathcal{F}$ that contains an $n$-vertex path.  Then $G$ contains
an induced path of size $\alpha (\log n -
\log(\alpha (\log n)^\beta))^\beta=(\alpha -o(1))(\log n)^\beta$.
\end{lemma}

\begin{proof}
Let $P$ be an $n$-vertex path in $G$.  Let $T$ be the block tree of
$G$.  Let $k$ be the number of blocks intersecting $P$.  If $k
\leq \alpha (\log n)^\beta$, then there is a block that contains a
subpath of $P$ with at least $\frac{n}{\alpha (\log n)^\beta}$
vertices.  By the hypothesis this block has an induced path of size
$\alpha \log(\frac{n}{\alpha (\log n)^\beta})^\beta = \alpha (\log n -
\log(\alpha (\log n)^\beta))^\beta=(\alpha -o(1))(\log n)^\beta$.  On
the other hand, if $k > \alpha (\log n)^\beta$, then there is a
path of $k$ blocks in $T$, which means that there are $k$ distinct
blocks $B_1, \dots B_k$ such that $B_i$ has exactly one vertex $v_i$
in common with $B_{i+1}$, and the vertices $v_1,\ldots,v_{k-1}$ are
pairwise distinct, and $v_iv_j$ is not an edge whenever $|i-j|\ge 2$.
For each $i=1,\ldots,k-1$ let $P_i$ be a shortest path between $v_i$
and $v_{i+1}$ in $B_{i+1}$.  Then we obtain a path
$v_1$-$P_1$-$v_2$-$P_2$-$\dots$-$v_{k-1}$ of size at least $\alpha
(\log n)^\beta$.
\end{proof}

We now consider partial $2$-trees.  Recall that every $2$-tree is a
planar graph (because a simplicial vertex of degree $2$ can be added
to any planar graph in a way that preserves planarity).

%%%%%%THEOREME DANS LES 2 ARBRES PARTIELS%%%%%
\begin{theorem}\label{partial}
If $G$ is a 2-connected partial $2$-tree that contains an $n$-vertex path, then
$G$ contains an induced path of size $\frac{\log (n-3)}{2}$.
\end{theorem}

%%%%%%PREUVE DANS LES 2 ARBRES PARTIELS%%%%%
\begin{proof}
Let $P$ be an $n$-vertex path in $G$. We add edges to $G$ in order
to obtain a $2$-tree $G'$.  As before, we can assume that $G'$ is a
minimal 2-tree containing $P$.

Using the proof of Theorem \ref{ktree}, we
can find in $T({G'})$ a path $P'$ of size $\ell\ge\log(n-3)$.  We
denote by $U$ the set of vertices of ${G'}$ corresponding to the
vertices of $P'$ (i.e., $U$ consists of the union of the cliques $K_i$
on 2 vertices defined in the proof of Theorem \ref{ktree}). The
subgraph of $G'$ induced by $U$ is denoted by $G'[U]$.

Given a planar embedding of a connected planar graph $G$, we define
the \emph{dual} graph $G^*$ of $G$ as follows: the vertices of $G^*$
are the faces of $G$, and two vertices $v_1^*$ and $v_2^*$ of $G^*$
are adjacent if and only if the corresponding faces of $G$ share an
edge.  We call \emph{weak dual} of $G$ the graph obtained from $G^*$
by deleting the vertex that represents the external face of $G$.

We call \emph{path of triangles} a $2$-tree having a plane embedding
whose weak dual is a path. It directly follows from our definition of the
cliques $K_i$ in the proof of Theorem \ref{ktree} that $G'[U]$ is a
path of triangles. From now on we fix a planar embedding of $G'[U]$
such that its weak dual is a path.

Since $G'[U]$ is a path of triangles, it has exactly two simplicial
vertices $a,b$, both of degree~$2$, and the other vertices of $G'[U]$ are
not simplicial.  Following the proof of Theorem~\ref{ktree}, the outer
face of $G'[U]$ can be partitioned into two paths going from $a$ to $b$,
and these two paths are induced paths in ${G'}$, so one of them has
size at least $\frac{\log (n-3)}{2}$.  We denote this path by $P_{G'}$.

Since $G$ is a subgraph of $G'$, some edges of $P_{G'}$ may not be in $G$;
we call them \emph{missing edges}.  Consider any missing edge $uu'$ of
$P_{G'}$, and let $w$ be the third vertex of the triangular inner face
of $G'[U]$ incident with $uu'$.  Since $G$ is $2$-connected, there is a
path $P_{uu'}$ in $G$ between $u$ and $u'$ avoiding $w$. The internal
vertices of such a path are necessarily disjoint from $P_{G'}$, since
otherwise $G'$ would contain $K_4$ as a minor (and it is well known
that any 2-tree is $K_4$ minor-free). We can assume without loss of
generality that $P_{uu'}$ is an induced path, by taking a shortest
path with the aforementioned properties. Using again that
$G'$ does not contain $K_4$ as a minor, it is easy to see that if 
$uu'$ and $vv'$ are
two missing edges, then the two paths $P_{uu'}$ and $P_{vv'}$ have no
internal vertex in common, no edges between their internal
vertices, and no edge from their internal vertices to $P_{G'}$ 
(except possibly to the endpoints of their
respective paths). Now, replacing
every missing edge $uv$ with the corresponding path $P_{uv}$, we get
an induced path $P_G$ that is at least as long as $P_{G'}$, and so $G$
contains an induced path of size $\frac{\log (n-3)}{2}$.
\end{proof}

As a direct consequence of Theorem~\ref{partial} and Lemma~\ref{2co},
we obtain:

\begin{corollary}\label{co:partial}
If $G$ is a partial $2$-tree that contains an $n$-vertex path, then
$G$ contains an induced path of size $(\frac{1}{2}-o(1))\log n$.
\end{corollary}

\section{Induced paths in planar and outerplanar graphs}

Since an outerplanar graph is a partial $2$-tree, we also obtain the
following corollary.

\begin{corollary}\label{outer}
If $G$ is an outerplanar graph with an $n$-vertex path, then $G$
contains an induced path of size $\frac{\log n}{2}(1-o(1))$.
\end{corollary}

We can give an alternative proof of this corollary.  We give the proof
in the case where the graph is $2$-connected.  If it is not, we can
use the Lemma \ref{2co}.  This proof is quite similar to the proof in
\cite{AV}.

\begin{theorem}\label{2coout}
If $G$ is a $2$-connected outerplanar graph with $n$ vertices, then
$G$ contains an induced path of size $\frac{\log n}{2}$.
\end{theorem}

\begin{proof}
Let $G$ be a $2$-connected outerplanar graph with $n$ vertices. We 
add edges to $G$ in order to
obtain a maximal outerplanar graph $G'$. We denote by $D$ and $D'$ the
weak duals of $G$ and $G'$, respectively. Note that $D$ and $D'$ are trees.

Each face of $G$ with $k$ vertices contains $k-2$ triangular faces of
$G'$, and for each vertex in $D$ corresponding to a $k$-vertex face,
we have a tree with $k-2$ vertices in $D'$.

Let $d$ be the diameter of $D'$, and let $m$ be the number of vertices
of $D'$.  We have $m \geq n-2$.  Consider a leaf $v$ of $D'$: it has
degree $1$, and its neighbor $u$ has degree at most $3$ (since each
vertex of $D'$ has degree at most $3$), and therefore at most $2$
neighbors distinct from $v$.  Each vertex of $D'$ is reachable from 
$v$ with a path of at most
$d-1$ edges, so we have $m \leq 2^{d-2}$.  Then we have $n \leq
2^{d-2}+2$, so $d \geq \log n$, and so there is a path $P'$ of size
$\ell' \geq \log n$ in $D'$.

We associate with each vertex $v$ of $D$, corresponding to a
$k$-vertex face of $G$, a weight of $k-2$ (which is the number of vertices in
$D'$ in the tree corresponding to $v$).  The \emph{weight} of a path
$P$ of $D$ is defined as the sum of the weights of its vertices.  Then a path
$P$ in $D$ of weigth $w$ corresponds in $G$ to a path of faces (a
sequence of faces in which any two consecutive faces share an edge), 
with $w-2$ vertices.

Let $P' = u'_1$-$\cdots$-$u'_{\ell'}$.  Each vertex $u'_i$ corresponds
to a face $F'_i$ of $G'$, which corresponds to a face $F_j$ of $G$ and
a vertex $u_j$ of $D$.  Then, we have $P' = u'_{i_1}$-$\cdots$-$
u'_{i_2-1}$-$u'_{i_2}$-$u'_{i_2+1}$-$\cdots$-$u'_{i_3}$-$\dots$-$u'_{i_s}$-$
\dots$-$u'_{i_s+t}$ with $i_1=1$, $i_s+t=\ell'$ and for $a=1, \dots,
s$, the vertices $u'_{i_a}, \dots , u'_{i_{a+1}-1}$ corresponding to
$u_{j_a}$ in $D$.  Let $P=u_{j_1}$-$\cdots$-$u_{j_s}$.  For each $a=1,
\dots, s-1$, $u_{j_a}$ is adjacent to $u_{j_{a+1}}$ because
$u'_{i_{a+1}-1}$ and $u'_{i_{a+1}}$ are adjacent in $D'$.  Moreover,
we claim that each vertex is present only once in $P$.  Suppose not.
Then we denote by $u_{j_a}$, corresponding to a face $F_a$ of $G$, a
vertex which is present several times in $P$, and by $b$ the smallest
index larger than $a$ such that $u_{j_a} = u_{j_b}$.  In $D'$, there
is a path between $u'_{i_{a+1}-1}$ and $u'_{i_b}$ in the tree
corresponding to $u_{j_a}$, and a path
$u'_{i_{a+1}-1}$-$u'_{i_{a+1}}$-$ \cdots$-$u'_{i_b-1}$-$u'_{i_b}$,
with no vertex from the tree corresponding to $u_{j_a}$.  Then there
is a cycle in the tree $D'$, which is a contradiction.  Therefore $P$
is a path.

Let $\ell$ be the weight of $P$.  Then we have $\ell \geq \ell'$,
because each vertex $u$ of $P$, corresponding to a $k$-vertex face,
has a weight $k-2$, and comes from $k' \leq k-2$ vertices in $D'$.
Therefore, $\ell \geq \log n$.

In $G$, $P$ corresponds to a path of faces separated by edges, with
$\ell+2$ vertices.  If we remove a vertex from each extremal face of
$P$, we get two induced paths in $G$, so one of them has size
$\frac{\ell}{2} \geq \frac{\log n}{2}$.
\end{proof}

This bound is optimal up to a constant multiplicative factor, as we
have seen before (recall that the graphs $G_i$ depicted in
Figure \ref{outerplanar} have $n=3\times2^{i-1}$ vertices
and their longest induced path have size $2(i+1)=2 \log
n+(2-2 \log 3)$).

The following theorem, proved in \cite{GLM}, gives a bound on the
largest induced outerplanar graph in a $3$-connected planar graph.
Their proof uses the existence of so-called ``{Schnyder woods}'' to
define some partial orders, followed by an application of Dilworth
theorem on these partial orders.

A \emph{bracelet} is a connected outerplanar graph where each
cut-vertex is shared by two blocks, and each block contains at most
two cut-vertices.

\begin{theorem}[\cite{GLM}]\label{dglm}
Any $3$-connected planar graph with $n$ vertices contains an induced
bracelet with at least $\sqrt[3]{n}$ vertices.
\end{theorem}

Using our Theorem \ref{2coout}, we can prove the following bound for a
bracelet with $n$ vertices.

\begin{lemma}\label{bra}
If $G$ is a bracelet containing $n$ vertices, then it contains an
induced path of size $\frac{1}{2}(\log n-\log \log n)=(\frac{1}{2}-o(1)) \log n$.
\end{lemma}

\begin{proof}
Denote by $k$ the number of blocks of $G$.

If $k \leq \log n$, then there is a block with at least $\frac{n}{\log
n}$ vertices.  It follows from Theorem \ref{2coout} that in this block
we can find an induced path with $\frac{1}{2} \log (\frac{n}{\log n})
= \frac{1}{2}(\log n-\log \log n)$ vertices.

If $k > \log n$, then $G$ has at least $k-1 \geq \log n -1$
cut-vertices.  Since $G$ is a bracelet, there are $k$ blocks $B_1,
\dots B_k$, with $B_i$ sharing a cut-vertex $c_i$ with $B_{i+1}$.  In
each block $B_i$, we take a shortest path (which is induced) from
$c_{i-1}$ to $c_i$; then the union of these paths is an induced path
of length at least $\log n$.
\end{proof}

Using Theorem~\ref{dglm}, Di Giacomo et al.~proved that a
$3$-connected planar graph with $n$ vertices contains an induced path
of size $\Omega(\frac{\log n}{\log \log n})$.  Combining Theorem
\ref{dglm} with Lemma~\ref{bra}, we obtain the following theorem:

\begin{theorem}\label{planar}
If $G$ is a $3$-connected planar graph with $n$ vertices, then $G$
contains an induced path of size $\tfrac12(\tfrac13 \log n-\log \log n)=(\frac{1}{6}-o(1)) \log n$.
\end{theorem}

\begin{proof}
Let $G$ be a 3-connected planar graph.  By Theorem \ref{dglm}, it
contains an induced bracelet $H$ with $m \geq \sqrt[3]{n}$ vertices.
By Lemma \ref{bra}, $H$ (and then $G$) contains an induced path of size  $\frac{1}{2}(\log m-\log \log m)>\tfrac12(\tfrac13 \log n-\log \log n)$.
\end{proof}

Our bound is (asymptotically) optimal up to a constant multiplicative factor, as shown
by the family of graphs $G_i$ depicted in Figure \ref{tri}.  The graph
$G_0$ is a triangle, and we obtain $G_i$ from $G_{i-1}$ by adding a
vertex adjacent to each triangle of $G_{i-1}$ that is not in
$G_{i-2}$.  The graph $G_i$ has $n=3+\frac{3^i-1}{2}$ vertices and its
longest induced path contains $\ell=i+1\geq \frac{\log n}{\log 3}$
vertices.

\begin{figure}[h]
   \begin{center}
   \begin{tikzpicture}[scale=0.3]
     \begin{scope}
\node[sommetn] (1) at (90:4) {};
\node[sommetn] (2) at (210:4) {};
\node[sommetn] (3) at (330:4) {};
\draw (1) -- (2) -- (3) -- (1);
\node at (270:4) {$G_0$};
\end{scope}
   \begin{scope}[xshift=250]
\node[sommetn] (1) at (90:4) {};
\node[sommetn] (2) at (210:4) {};
\node[sommetn] (3) at (330:4) {};
\node[sommetn] (4) at (0,0) {};
\draw (1) -- (2) -- (3) -- (1);
\draw (1) -- (4) -- (2);
\draw (3) -- (4);
\draw (1) --(4);
\node at (270:4) {$G_1$};
\end{scope}
   \begin{scope}[xshift=500]
\node[sommetn] (1) at (90:4) {};
\node[sommetn] (2) at (210:4) {};
\node[sommetn] (3) at (330:4) {};
\node[sommetn] (4) at (0,0) {};
\node[sommetn] (11) at (30:1.2) {};
\node[sommetn] (22) at (150:1.2) {};
\node[sommetn] (33) at (270:1.2) {};
\draw (1) -- (2) -- (3) -- (1);
\draw (1) -- (4) -- (2);
\draw (3) -- (4);
\draw (1) -- (11) -- (4) -- (22) -- (2);
\draw (1) -- (22);
\draw (2) --  (33) --  (3) -- (11);
\draw (4) -- (33);
\draw (1) --(4) -- (33);
\node at (270:4) {$G_2$};
\end{scope}
   \end{tikzpicture}
   \end{center}
   \caption{A family of triangulations}
   \label{tri}
\end{figure}
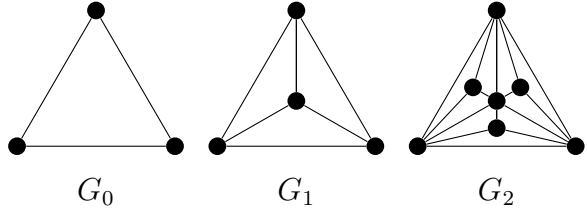

 From our bound for $3$-connected planar graphs, we will now deduce
a bound for $2$-connected planar graphs.  Since such graphs do not
necessarily contain long paths (see for example the complete bipartite
graph $K_{2,n}$), we restrict ourselves to 2-connected planar graphs
with long paths.  We will use the so-called \emph{SPQR-trees}
\cite{BT}, defined as follows.

Let $G$ be a $2$-connected graph.  One can represent the interaction
of $3$-connected induced subgraphs of $G$ by a tree $T_G$, in which
each node is associated to a subgraph and has one of four types:
\begin{itemize}
\item
Each node of type S is associated with a cycle on at least three
vertices;
\item
Each node of type R is associated with a $3$-connected simple
subgraph;
\item
Each node of type P is associated with two vertices, with three or
more edges between them (and two nodes of type P are not adjacent in
$T_G$);
\item
Each node of type Q is associated with a single edge.  This case is
used only when the graph has only one edge.
\end{itemize}
If $x$ and $y$ are two adjacent nodes of $T_G$, and $G_x$ and $G_y$
are the associated graphs, then the edge $xy$ of $T_G$ is associated
to one pair of adjacent vertices in $G_x$, and one pair of adjacent
vertices in $G_y$ (the edges between these pairs of vertices are
called \emph{virtual edges}).  Given an SPQR-tree $T$, we obtain the
corresponding $2$-connected graph as follows.  For each edge $xy$ in
$T_G$, we do the following: let $(a,b)$ be the pair of adjacent
vertices associated to $xy$ in $G_x$, and let $(c,d)$ be the pair of
adjacent vertices associated to $xy$ in $G_y$, then we identify the
vertices $a$ and $c$, and the vertices $b$ and $d$, and remove the
virtual edges between the two newly created vertices (see
Figure~\ref{tri}).  For any $2$-connected graph $G$ the SPQR-tree
$T_G$ is unique up to isomorphism.

Given a subtree $T'$ of $T_G$, we can define an induced subgraph
$G_{T'}$ of $G'$ as described above by identifying vertices and then
removing all virtual edges (including those that are not matched).

\begin{figure}[h]
   \begin{center}
   \begin{tikzpicture}[scale=1.7]
     \begin{scope}[rotate=90]
\node[sommetn] (1) at (0,0) {};
\node[sommetn] (2) at (0,-1) {};
\node[sommetn] (3) at (1,0) {};
\node[sommetn] (4) at (1,-1) {};
\node[sommetn] (5) at (60:1) {};
\node[sommetn] (6) at (30:0.58) {};
\draw (1) -- (2) -- (4) -- (3) -- (1);
\draw (1) -- (5) -- (3);
\draw (1) -- (6) -- (5);
\draw (6) -- (3);
\end{scope}

   \begin{scope}[xshift=100, rotate=90]]
\node[sommetn] (1) at (0,0) {};
\node[sommetn] (3) at (1,0) {};
\node[sommetn] (5) at (60:1) {};
\node[sommetn] (6) at (30:0.58) {};
\node[minimum width=2.5cm, minimum height=3cm,
rectangle,rounded corners=10pt,draw] at (0.5,0.4) {};
\node at (1.5,0.4) {R};
\draw (1) -- (5) -- (3);
\draw (1) -- (6) -- (5);
\draw (6) -- (3);
\draw[dashed] (1) -- (3);
\node[sommetn] (1') at (0,-0.8) {};
\node[sommetn] (3') at (1,-0.8) {};
\draw (1')--(3');
\node[minimum width=1.3cm, minimum height=3cm,
rectangle,rounded corners=10pt,draw] at (0.5,-0.8) {};
\node at (1.5,-0.8) {P};
\draw[dashed] (1') to[bend left] (3');
\draw[dashed] (1') to[bend right] (3');
\node[sommetn] (1'') at (0,-1.6) {};
\node[sommetn] (2) at (0,-2.6) {};
\node[sommetn] (3'') at (1,-1.6) {};
\node[sommetn] (4) at (1,-2.6) {};
\draw (1'') -- (2) -- (4) -- (3'');
\draw[dashed] (1'') -- (3'');
\node[minimum width=2.5cm, minimum height=3cm,
rectangle,rounded corners=10pt,draw] at (0.5,-2.1) {};
\node at (1.5,-2.1) {S};
\node (a) at (0.5,-0.2) {};
\node (b) at (0.5,-0.5) {};
\node (c) at (0.5,-1.1) {};
\node (d) at (0.5,-1.5) {};
\draw[very thick] (a) -- (b);
\draw[very thick] (c) -- (d);
\end{scope}

   \end{tikzpicture}
   \end{center}
   \caption{A graph and the associated SPQR-tree}
   \label{tri}
\end{figure}
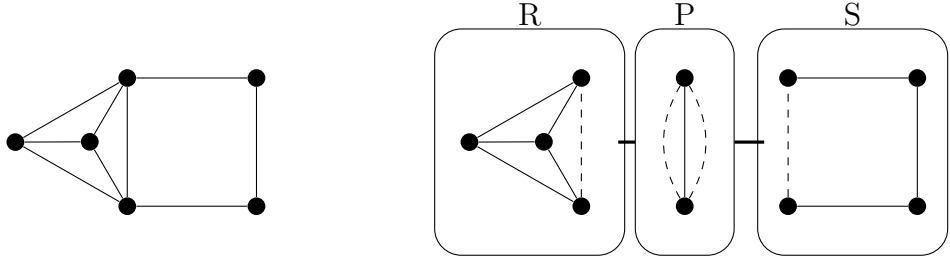

\begin{theorem}\label{planar2}
If $G$ is a $2$-connected planar graph containing a path with $n$
vertices, then $G$ contains an induced path of size at least $\frac{\sqrt{\log
n}}{2\sqrt{6}}-\tfrac14 \log \log n-1=\frac{\sqrt{\log
n}}{2\sqrt{6}}(1-o(1))$.
\end{theorem}

\begin{proof}
Let $P$ be a path on $n$ vertices in $G$.  We consider the smallest
induced subgraph $G'$ of $G$ which contains $P$ and is $2$-connected.
Let $T_{G'}$ be the SPQR-tree corresponding to $G'$.

Let $\alpha = 2^{\sqrt{\frac{3 \log n}{2}}}$.  Note that either there
is a node of $T_{G'}$ whose associated graph has size $\alpha$, or
every graph associated with a node of $T_{G'}$ has less than $\alpha$
vertices.

Suppose first that there is a node $x$ of $T_{G'}$ whose associated
graph has size $\alpha$.  Then $x$ is a node of type S or R. If $x$ is
a node of type S, then there is an induced path of size $\alpha-1$ in
the associated graph (which is a cycle).  If $x$ is a node of type R,
then the associated graph is a $3$-connected planar graph, and by
Theorem \ref{planar}, there is an induced path of size $\frac1{6}\log
\alpha-\tfrac12 \log \log \alpha$ in the associated graph.  In both cases, we have
an induced path $P_x$ of size at least $\frac1{6}\log
\alpha-\tfrac12 \log \log \alpha$
in the graph associated with the node $x$.  If $P_x$ is also an
induced path in $G$ we are done, so assume the contrary, which means
that $P_x$ contains some virtual edges.  Let $ab$ be any virtual edge
in $P_x$.  Then $ab$ corresponds to a virtual edge $cd$ in some node
$y$ adjacent to $x$ in $T_{G'}$.  Let $P_{ab}$ be a shortest path from
$c$ to $d$ in the subgraph $G_{T_y}$, where $T_y$ is the subtree of
$T_{G'}$ rooted at $y$ ($T_y$ contains the descendance of $y$, viewing
$x$ as the root of $T_{G'}$).  In $P_x$ we replace each virtual edge
$ab$ with the corresponding path $P_{ab}$, so we denote by $P'$ the
resulting path of $G'$.  Observe that this path $P'$ is an induced
path in $G'$, because all the replacement paths are in distinct
subtrees of $T_{G'}$ (each of these subtrees corresponds to a distinct
neighbor of $x$ in $T_{G'}$).  Hence $P'$ is an induced path in $G'$,
of size at least the size of $P_x$.  So we have an induced path of
size $\frac1{6}\log
\alpha-\tfrac12 \log \log \alpha$ in $G'$, which is an induced path
of size $\frac1{6}\log
\alpha-\tfrac12 \log \log \alpha\ge \frac{\sqrt{\log
n}}{2\sqrt{6}}-\tfrac14 \log \log n-1$ in $G$.

Suppose now that every graph associated with a node of $T_{G'}$ has
less than $\alpha$ vertices.  Then there are at least
$\frac{n}{\alpha}$ nodes in $T_{G'}$.  Since each graph corresponding
to a node of type R or S is a planar graph with no multiple edge, and
has at most $\alpha$ vertices, it contains at most $3 \alpha - 6$
edges.  So the degree of a node of type R or S is at most $3 \alpha -
6$, since each edge contributes to at most $1$ in the degree, if it is
a virtual edge.  Concerning the nodes of type $P$, we claim that their
degree in $T_{G'}$ is at most $3$.  For suppose that there is a node
$X$ of type P of degree at least $4$.  Since the removal of the two
vertices in the graph associated with $X$ disconnects the graph $G'$,
there are edges of the path $P$ in at most three components associated
with nodes adjacent to $X$.  Since these nodes are adjacent to $X$,
they are not of type P, and so they have at least three vertices.
Removing in $G'$ the vertices of those components that do not
intersect $P$, we obtain a smaller graph, induced, $2$-connected, and
containing the path $P$, contradicting the minimality of $G'$.  So the
claim holds.  It follows that the degree of every node in $T_{G'}$ is
at most $3\alpha$.  Let $d$ be the diameter of $T_{G'}$.  Then we have
$\frac{n}{\alpha} \leq (3\alpha)^{(d-2)}$.  So there is a path in
$T_{G'}$ of size $d \geq \frac{\log n}{\log (3\alpha)}-\frac{\log
\alpha}{\log (3 \alpha)}+2 \ge \frac{\log n}{\log (3\alpha)}$.

\medskip

We claim that if there is a path $\mathcal{P}$ of size $\ell$ in
$T_{G'}$, then there is an induced path in $G'$ of size
$\frac{\ell}{4}$.  First, since two nodes of type P cannot be
adjacent, there are at most $\frac{\ell}{2}$ nodes of type P in
$\mathcal{P}$.  The other nodes are of type R or S. Denote by $p_1,
\dots, p_\ell$ the nodes of $\mathcal{P}$ and $e_1, \dots, e_{\ell-1}$
its edges, where $e_i=p_i p_{i+1}$.  Each edge $e_i$ corresponds to
one virtual edge in $p_i$ and one virtual edge in $p_{i+1}$, which
correspond to two vertices $x_i$, $y_i$ (adjacent or not) in the graph
$G'$.  We have $\{x_i, y_i\} = \{x_{i+1}, y_{i+1}\}$ if and only if
the node $p_{i+1}$ of $\mathcal{P}$ is of type P. Denote by $p_{i_1},
\ldots, p_{i_k}$ the nodes that are not of type P. We have $k \geq
\frac{\ell}{2}$, since at most $\frac{\ell}{2}$ nodes have type P. For
each $j$, there is at most one vertex in common between $\{x_{i_j},
y_{i_j}\}$ and $\{x_{i_{j+1}}, y_{i_{j+1}}\}$.  Then we keep the name
of $\{x_{i_1}, y_{i_1}\}$, and rename the others vertices so that if
$\{x_{i_j}, y_{i_j}\}$ and $\{x_{i_{j+1}}, y_{i_{j+1}}\}$ have a
vertex in common, then we have either $x_{i_j} = x_{i_{j+1}}$ or
$y_{i_j} = y_{i_{j+1}}$.  In total, there are at least
$\frac{\ell}{2}$ vertices $x_{i_j}$ and $y_{i_j}$, so one of these two
sets, say the set $\{x_{i_j}\mid 1\le j \le k\}$, contains at least
$\frac{\ell}{4}$ elements.  We can then find an induced path in $G'$
containing these vertices: we consider the induced subgraph of $G$
corresponding to the subtree rooted at $p_{i_{j+1}}$ and containing
$p_{i_{j+1}}$ and its descendance, except $p_{i_j}$, $p_{i_{j+2}}$ and
their descendance, and take a shortest path between $x_{i_j}$ and
$x_{i_{j+1}}$ in this graph.  The path obtained is induced since each
path is taken in subtrees having no vertex in common.  Then, we obtain
an induced path of size $\frac{\log n}{4 \log (3\alpha)}
\ge \frac{\sqrt{\log n}}{2\sqrt{6}}-1$ in $G'$.
\end{proof}

Using Theorem \ref{planar2}, we deduce the following corollary for
connected planar graphs using Lemma \ref{2co}.

\begin{corollary}\label{planarco}
If $G$ is a connected planar graph containing a path with $n$
vertices, then $G$ contains an induced path of size $\frac{\sqrt{\log
n}}{2\sqrt{6}}(1-o(1))$.
\end{corollary}

%\begin{proof} Let $G$ be a connected planar graph containing a path
%with $n$ vertices.  We consider the tree of $2$-connected components
%of $G$.  The vertices of the path are in $k$ components.  Either $k
%\geq \sqrt{\log n}$ and then we get a path of size $\sqrt{\log n}$,
%or we have a $2$-connected component with at least $\frac{n}{\sqrt
%{\log n}}$ vertices, in which we have, with Theorem \ref{planar2}, a
%path of size $\frac{\sqrt{\log \frac{n}{\sqrt {\log n}}}}{6\sqrt{2}}
%= \frac{\sqrt{\log n}}{6\sqrt{2}}(1-o(1))$.  \end{proof}

In this paper, a {\em surface} is a non-null compact connected 2-manifold without boundary.  A surface is either orientable or non-orientable. 
The \emph{orientable surface~$\OS_h$ of genus~$h$} is obtained by adding $h\ge0$
\emph{handles} to the sphere, and the \emph{non-orientable surface~$\NOS_k$ of genus~$k$} is formed by adding $k\ge1$ \emph{cross-caps} to the sphere. 
The {\em Euler genus} of a surface $\Sigma$ is defined as twice its genus if $\Sigma$ is orientable, and as its genus if $\Sigma$ is non-orientable. We refer the reader to the monograph of Mohar and
Thomassen~\cite{MoTh} for background on graphs on surfaces.

Using Corollary~\ref{planarco}, we easily deduce a similar bound for graphs
embedded on a fixed surface.

\begin{theorem}\label{surf}
For any surface $\Sigma$, any graph $G$ embedded in $\Sigma$, and
containing a path with $n$ vertices, also contains an induced path of
size $(\frac{1}{6\sqrt{2}}-o(1))\sqrt{\log n}$ (where the $o(1)$
depends on $\Sigma$).
\end{theorem}

\begin{proof}
Let $f_g$ be the function defined as follows: $f_0$ is the $o(1)$
defined in Corollary \ref{planarco}, and for each $g>0$,
$f_g(n)=\tfrac{1}{2\sqrt{6}}-(\tfrac{1}{2\sqrt{6}}-f_{g-1}(\tfrac{n}{\log
n}))(1-\tfrac{\log \log n}{ \log n})$.  It is not difficult to prove
by induction on $g$ that for fixed $g$, $f_g=o(1)$.  We prove by
induction on the Euler genus $g$ of $\Sigma$ that every graph
embeddable in $\Sigma$ with a path $P$ on $n$ vertices has an induced
path on $(\tfrac{1}{2\sqrt{6}}-f_g(n))\sqrt{\log n}$ vertices.

If $g=0$, the result follows from Corollary \ref{planarco}, so assume
that $g>0$.  Let $\mathcal{C}$ be a shortest non-contractible cycle of
$G$.  Note that $\mathcal{C}$ is an induced cycle, therefore, if
$\mathcal{C}$ has size at least $\log n$, then $G$ contains an induced path of
size $\log n-1\ge (\tfrac{1}{2\sqrt{6}}-f_g(n))\sqrt{\log n}$ and we
are done.  Hence, we can assume that $\mathcal{C}$ contains at most
$\log n$ vertices.

The path $P$ and the cycle $\mathcal{C}$ can have at most $\log n$
vertices in common.  Let us denote by $p_1, \dots, p_k$ these common
vertices, in order of appearance in $P$.  Then we have $P =
P_0$-$p_1$-$P_1$-$p_2$-$\cdots $-$p_k$-$P_k$, where each $P_i$ is a
path (possibly empty).  Since $P$ has $n$ vertices, there is one $P_i$
with at least $\frac{n}{\log n}$ vertices.  If we remove
the vertices of $\mathcal{C}$ from $G$, we obtain a graph $G'$ such that each connected
component is embeddable on a surface of Euler genus at most $g-1$, and
at least one such component contains a path on $\frac{n}{\log n}$
vertices.  Then, by induction, $G'$ (and therefore $G$) contains an
induced path of size
\begin{eqnarray*}
& & (\tfrac{1}{2\sqrt{6}}-f_{g-1}(\tfrac{n}{\log n}))\sqrt{\log 
\tfrac{n}{\log n}}\\
&\ge & (\tfrac{1}{2\sqrt{6}}-f_{g-1}(\tfrac{n}{\log n}))
   (1- \tfrac{\log \log n}{\log n}) \sqrt{\log n}\\
&= & (\tfrac{1}{2\sqrt{6}}-f_{g}(n))\sqrt{\log n},
\end{eqnarray*}
as desired.
\end{proof}

We do not know if the bound in Theorem~\ref{planar2} and
Corollary~\ref{planarco} is optimal.  We now construct a family of
planar graphs containing a path with $n$ vertices in which the longest
induced path has size $3\frac{\log n}{\log \log n}$.  Let $G_1$ be the
graph obtained by taking a path $P=p_1$-$ \cdots$-$p_k$ on $k$
vertices and adding two adjacent vertices $u$ and $v$ that are
adjacent to each vertex of the path (see Figure \ref{2copla}).  The
graph $G_1$ has $k+2$ vertices and a Hamiltonian path $u, p_1, \dots,
p_k, v$, unique up to symmetry.  We define $G_{i+1}$ by induction: we
sart with a copy of $G_1$ (called the \emph{original copy} of $G_1$)
and replace each edge $p_j p_{j+1}$ of the path $P$ in $G_1$ by a copy
of $G_i$, identifying $u$ of $G_i$ with $p_j$ of $G_1$ and $v$ of
$G_i$ with $p_{j+1}$ of $G_1$.  The vertices $u$ and $v$ in $G_{i+1}$
are then defined to be the vertices $u$ and $v$ or the original copy
of $G_1$.  We claim that $G_i$ has at least $(k-2)^{i-1}$ vertices, a
Hamiltonian path, and that the longest induced path in $G_i$ has $2i +
(k-2)$ vertices.

\begin{figure}[h]
   \begin{center}
   \begin{tikzpicture}[scale=0.6]
     \begin{scope}
\node[sommetn] (1) at (90:4) {};
\node[sommetn] (2) at (210:4) {};
\node at (210:4.8) {$u$};
\node at (330:4.8) {$v$};
\node[sommetn] (3) at (330:4) {};
\node[sommetn] (4) at (90:1) {};
\node[sommetn] (7) at (90:2) {};
\node[sommetn] (8) at (90:3) {};
\node[sommetn] (6) at (0,0) {};
\node[sommetn] (5) at (270:1) {};
\draw (1) -- (2) -- (3) -- (1);
\draw (1) -- (4) -- (2);
\draw (3) -- (4);
\draw (4) -- (5);
\draw (1) --(4) -- (5);
\draw (2) -- (6) -- (3);
\draw (2) -- (5) -- (3);
\draw (2) -- (7) -- (3);
\draw (2) -- (8) -- (3);
% \node at (270:4) {$G_1$};
\end{scope}

   \end{tikzpicture}
   \end{center}
   \caption{The graph $G_1$}
   \label{2copla}
\end{figure}
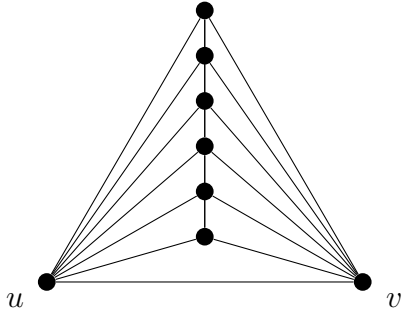

First, note that in $G_i$ the longest induced path starting from $u$
or $v$ has size $i+1$.  This is trivial for $G_1$, and if it is true
for $G_{i-1}$, then in $G_i$, the longest path starting by $u$ (or
$v$) is obtained by taking an edge from $u$ (or $v$) to a vertex of
$P$, and then taking the longest induced path starting by this vertex
in the copy of $G_{i-1}$, which by induction has $i-1$ vertices.

Now, observe that an induced path in $G_i$ consists of an induced path
in some copy of $G_{i-1}$, followed by an induced path in $G_1$,
followed by an induced path in some copy of $G_{i-1}$.  Note that the
two copies might coincide, and if the induced path is not completely
contained in a unique copy of $G_{i-1}$, then it contains some vertex
$u$ or $v$ of the copies of $G_{i-1}$ it intersects.  In any case, any
induced path in $G_i$ contains at most $2i+(k-2)$ vertices.

For $k=i$, we have a $2$-connected planar graph with a path on $n\ge
(k-2)^{k-1}$ vertices and a longest induced path of size $3k-2\le
3\frac{\log n}{\log \log n}$.

We can use a similar construction to find a family of Hamiltonian
chordal graphs of maximum clique size $2t+1$ (and therefore tree-width
$2t$) with $n$ vertices and no longest induced path of length more
than $2t(\log n)^{\frac{1}{t}}$.

We first deal with graphs of tree-width $4$.  We consider the
outerplanar graphs of Figure~\ref{outerplanar}.  We build $G_{1,4}$ by
taking some $G_i$ of Figure~\ref{outerplanar} (a graph with $n$
vertices and a longest induced path of length $2 \log n$), and we add
two adjacent vertices $u,v$ that are adjacent to every vertex of
$G_i$.  The graph $G_{1,4}$ is a $4$-tree and contains a Hamiltonian
path $P$ starting at $u$ and ending at $v$.  Then we obtain
$G_{k+1,4}$ by replacing each edge $a,b$ in the Hamiltonian path $P$
in $G_{1,4}$ by $G_{k,4}$, identifying $a$ with $u$ and $b$ with $v$
(there is still a Hamiltonian path starting by $u$ and ending by $v$
in $G_{k+1,4}$, and the graph has tree-width $4$).  We claim that
$G_{k,4}$ has at least $n^k$ vertices and a longest induced path of
length at most $2(\log n+k-1)$.

First, note that in $G_{k,4}$ the longest induced path starting from
$u$ or $v$ has size $k+1$.  Then observe as above that induced paths
in $G_{k,4}$ are the concatenation of an induced path in a copy of
$G_{k-1,4}$, an induced path in $G_{1,4}$, and an induced path in a
copy of $G_{k-1,4}$.  As before, if the induced path of $G_{k,4}$ is
not contained in a copy of $G_{k-1,4}$, then it contains a vertex $u$
or $v$ of each of the at most copies of $G_{k-1,4}$ it intersects.
Again, we conclude that any induced path in $G_{k,4}$ has size at most
$2(\log n +k-1)$.

For $k=\log n$, we obtain a graph with $N\ge n^{\log n}$ vertices,
with a longest induced path of length at most $4 \log n \le 4 (\log
N)^{\frac{1}{2}}$.

If we have a family of Hamiltonian graphs of tree-width $2t$ with $N$
vertices and a longest induced path of length $(\log
N)^{\frac{1}{t}}$, then we can build the family for tree-width
$2(t+1)$.  We take a Hamiltonian graph $G$ of tree-width $2t$ with $n$
vertices and a longest induced path of length $(\log
n)^{\frac{1}{t}}$, and we add two adjacent vertices $u,v$ that are
adjacent to every vertex of $G$: denote by $G_{1,2(t+1)}$ this graph.
Then we obtain $G_{k+1,2(t+1)}$ by replacing each edge $ab$ in the
Hamiltonian path in $G_{1,2(t+1)}$ by a copy of $G_{k,2(t+1)}$,
identifying $a$ with $u$ and $b$ with $v$, which gives a graph of
tree-width $2(t+1)$.

Similarly, $G_{k,2(t+1)}$ has $N\ge n^k$ vertices, and a longest path
in $G_{k,2(t+1)}$ has size at most $2(k+t(\log n)^{\frac{1}{t+1}})$.
For $k=(\log n)^{\frac{1}{t+1}}$, we have $N \ge n^{(\log
n)^{\frac{1}{t+1}}}$ vertices and the longest path has size at most
$2(t+1)(\log N)^\frac{1}{t+1}$.

%%%%%%%
\section{Induced paths in interval graphs}

An important class of chordal graphs is the class of interval graphs.
An \emph{interval graph} is the intersection graph of a family of
intervals on the real line.  We will use the following notation.  Let
$G$ be an interval graph.  For every vertex $v\in V(G)$, let
$I(v)=[l(v),r(v)]$ be the corresponding interval in an interval
representation of $G$.  We may assume without loss of generality that
the real numbers $l(v),r(v)$ ($v\in V(G)$) are all different.  We call
\emph{left ordering} the ordering $v_1<\cdots<v_n$ of the vertices of
$G$ such that $v< w$ if and only if $l(v) < l(w)$.  It is easy to see
that for each $i$ the vertex $v_i$ is a simplicial vertex in the
subgraph induced by $v_1, \dots, v_i$.

We now prove that interval graphs satisfy Conjecture~\ref{conj:1}.

\begin{theorem}\label{thm:int}
For any integer $k$, there is a constant $c_k>0$ such that if $G$ is
an interval graph with $n$ vertices containing a Hamiltonian path and
$G$ has maximum clique size $k$, then $G$ has an induced path of size
at least $c_k (\log n)^{\frac{1}{(k-1)^2}}$.
\end{theorem}

The proof of this theorem is divided into three lemmas.

\begin{lemma}\label{lem:int1}
Let $G$ be an interval graph with $n$ vertices, containing a
Hamiltonian path.  Let $k\ge 2$ be the maximum clique size in $G$ and
let $v_1<\cdots<v_n$ be a left ordering of the vertices of $G$.  Then
$G$ contains an induced subgraph $H$ of size at least $f_1(n,k)=
\log_{k+2}(\tfrac{n}{(k+2)!})$ containing $v_n$ where, in the
induced left ordering, each vertex is adjacent to its successor.
\end{lemma}

\begin{proof}
We prove the lemma by induction on $n$ and $k$.  If $k=2$, then the
hypothesis implies that $G$ is a path on $n$ vertices, and the desired
result holds with $H=G$, and the inequality $n\ge f_1(n,2)=
\log_{4}(\tfrac{n}{6})$ holds for all $n\ge 2$.  If $n\le 3$, then
either $k=2$, or $k=3$ and since $f_1(3,3)<1$ the result holds
trivially.  Now assume that $k\ge 3$ and $n\ge 4$.  Let $v_1<\cdots<
v_n$ be a left ordering of $V(G)$.  Let $\mathcal{P}$ be a Hamiltonian
path in $G$.  Let $i$ be the largest integer such that $v_i$ is a
neighbor of $v_n$, and define the sets
\begin{quote}
$L=\{v\in V(G)\mid$ $r(v)<l(v_i)\}$, \\
$R=\{v\in V(G)\mid$ $l(v_i)<l(v)\}$, and \\
$K=\{v\in V(G)\mid$ $l(v)\le l(v_i)\le r(v)\}$.
\end{quote}
So $L$, $R$ and $K$ form a partition of $V(G)$.  Clearly $K$ is a
clique, so the subgraph $G_{RK}$ of $G$ induced by $R \cup K$ also has
a Hamiltonian path.  Every vertex $v$ in $R$ satisfies $l(v_i)<l(v)\le
l(v_n)\le r(v_i)$, so every vertex of $G_{RK}\setminus\{v_i\}$ is
adjacent to $v_i$.  It follows that $G_{RK}\setminus\{v_i\}$ has
maximum clique size at most $k-1$.  Observe that $v_n$ is the last
vertex of $G_{RK}\setminus\{v_i\}$ in its induced left ordering.
Therefore, if $G_{RK}\setminus\{v_i\}$ contains at least
$\tfrac{n}{k+2}$ vertices, then by the induction hypothesis
$G_{RK}\setminus\{v_i\}$ (and then $G$) contains an induced subgraph
that satisfies the desired property and has size at least
$f_1(\frac{n}{k+2},k-1)= \log_{k+1}(\frac{n}{(k+2) (k+1)!})=
\log_{k+1}(\tfrac{n}{(k+2)!})\geq f_1(n,k)$.

Assume now that $G_{RK}\setminus\{v_i\}$ contains strictly less that
$\tfrac{n}{k+2}$ vertices.  Then $L$ contains at least
$\tfrac{k+1}{k+2}\,n$ vertices, and the restriction of $\mathcal{P}$
to $L$ consists of at most $k+1$ subpaths (because $|K|\le k$), so one
of these subpaths has size at least $\tfrac{n}{k+2}$.  Let $G_L$ be
the graph induced by the vertices of this subpath, together with $v_i$
and a vertex of $K$ adjacent to an endpoint of the subpath.  Note that
$G_L$ is Hamiltonian, and $v_i$ is by definition the last vertex of
$G_L$ in its induced left ordering.  By the induction hypothesis $G_L$
(and then $G$) contains an induced subgraph $H$ that satisfies the
desired property and has size at least size $f_1(\frac{n}{k+2},k)$.
In particular, the last vertex in the induced left ordering of $H$ is
$v_i$.  Appending $v_n$ to $H$ yields an induced subgraph of $G$ that
satisfies the desired property and has size at least
$f_1(\frac{n}{k+2},k)+1= \log_{k+2}(\frac{n}{(k+2) (k+2)!})+1=
f_1(n,k)$.
\end{proof}

\begin{lemma}
Let $G$ be an interval graph with $n$ vertices, and let $k\ge 2$ be
the maximum clique size in $G$.  Suppose that in the left ordering
each vertex is adjacent to its successor.  Then $G$ contains an
induced subgraph $H$ of size $f_2(n,k)=n^{\frac{1}{k-1}}$ where in the
left ordering each vertex is adjacent to its successor and where there
is no simplicial vertices, except the last and the first.
\end{lemma}

\begin{proof}
First, note that when we remove a simplicial vertex, we still have a
graph where in left ordering, each vertex is adjacent to its
successor.

At each step, we will remove a simplicial vertex which is not the
first or the last, until there is no simplicial vertices other than
the last and the first.

Let $v$ be a vertex which is not the first or the last.  Denote by $w$
the first neighbor of $v$ in left ordering.  We claim that:
\begin{equation}\label{wnr}
\mbox{$w$ is never removed.}
\end{equation}
Indeed if $w$ is the first vertex in the ordering, it cannot be
removed; so let us assume that $w$ is not the first vertex, and let
$w'$ be its predecessor.  So $w'$ is adjacent to $w$, and $w'$ is not
adjacent to $v$ by the definition of $w$.  We claim that at each step
there is a non-edge $ab$ such that $a$ and $b$ are neighbors of $w$
and $l(a) < l(w)< l(v)\le l(b)$.  This is true at the first step with
$a=w'$ and $b=v$.  Assume that at step $i$, there is such a non-edge
$ab$.  We remove a simplicial vertex $s$.  Clearly $s\neq w$.  If
$s\notin\{a,b\}$, then $ab$ remains a non-edge in the neighborhood of
$w$.  Suppose that $s=b$.  Let $b'$ be the successor of $b$ (note that
$b'$ exists since $s$ is not the last vertex).  Then $b'$ is adjacent
to $w$ since $b$ is simplicial, and $l(v) \leq l(b)<l(b')$.  Also $b'$
is not adjacent to $a$, for that would force $b$ to be adjacent to $a$
(because $l(a)< l(b)< l(b')$).  Hence $ab'$ is a non-edge with the
desired property.  Finally suppose that $s=a$.  Let $a'$ be its
predecessor (note that $a'$ exists since $s$ is not the first vertex).
Then $a'$ is adjacent to $w$ since $a$ is simplicial.  Also $a'$ is
not adjacent to $b$, for otherwise $a'$ is adjacent to $v$ (because
$l(a')<l(w)<l(v)\le l(b)$), and so $a'$ contradicts the choice of $w$.
Hence $a'b$ is a non-edge with the desired property.  Thus (\ref{wnr})
holds.

\medskip

Now, we can prove the lemma by induction on $k$.  For $k=2$, the graph
$G$ is a path and we can take $H=G$.  Assume that $k\ge 3$.  For a
vertex $w$, we denote by $S_w$ the set of vertices having $w$ as their
first neighbor.  Suppose that for some vertex $w$, the set $S_w$ has
size at least $n^\frac{k-2}{k-1}$.  Note that all the vertices of
$S_w$ are consecutive, and the subgraph $G[S_w]$ of $G$ induced by
$S_w$ has clique size at most $k-1$; hence, by the induction
hypothesis, $G[S_w]$ has an induced subgraph of size
$n^{\frac{1}{k-1}}$ with the desired property.  Assume now that for
every $w$, $S_w$ has size at most $n^\frac{k-2}{k-1}$.  It follows
from (\ref{wnr}) that for every removed vertex $v$, the first neighbor
of $v$ is preserved.  Each first neighbor is counted at most
$n^\frac{k-2}{k-1}$ times, so at least
$n/n^\frac{k-2}{k-1}=n^{\frac{1}{k-1}}$ vertices are preserved, as
desired.
\end{proof}

\begin{lemma}
Let $G$ be an interval graph with $n$ vertices and with maximal clique
of size $k\ge 2$ where, in left ordering, each vertex is adjacent to
its successor and where there are no simplicial vertices, except the
last and the first.  Then $G$ contains an induced path of size
$f_3(n,k)=(\frac{n}{k})^{\frac{1}{k-1}}$.
\end{lemma}

\begin{proof}
Note that the hypothesis implies that $G$ is connected.  Let $S=\{s_1,
\dots, s_q\}$ be a maximal stable set, with $s_1< \cdots< s_q$, and
such that each interval $I(s_i)$ ($i\in\{1,\dots,q\}$) is minimal by
inclusion and (with respect to these two conditions) the numerical
vector $l_S= (l(s_1),\ldots,l(s_q))$ is minimal in lexicographic
order.  We claim that for each $i=2, \dots, q-1$, there is a vertex
$t_{i-1}$ with $r(s_{i-1}),l(s_i) \in I(t_{i-1})$.  This follows from
the fact that $s_i$ is not simplicial and therefore has a non-edge
$ab$, say with $a<b$, in its neighborhood.  By inclusion-wise
minimality of $I(s_i)$, $I(a)$ intersects $l(s_i)$, and if $a$ is not
adjacent to $s_{i-1}$ then $I(a)$ (or an interval contained in $I(a)$
which is minimal with this property) contradicts the lexicographic
minimality of $S$.  Moreover we claim that there is a vertex $t_{q-1}$
with $r(s_{q-1}),l(s_q) \in I(t_{q-1})$.  Indeed, since $G$ is
connected, there is a chordless path $u_0$-$\cdots$-$u_p$ with
$u_0=s_{q-1}$ and $u_p=s_q$.  Let $j$ be the smallest integer in
$\{0,\ldots,p\}$ such that $l(u_j)>r(u_0)$; note that $j$ exists since
$l(u_p)>r(u_0)$.  Moreover, $l(u_j)\le l(u_p)$ since the path is
chordless.  In fact $j=p$, for otherwise we have $l(u_j)< l(u_p)$ and
the set $\{s_1, \ldots, s_{q-1}, u_j\}$ would contradict the choice of
$S$.  Now the vertex $u_{p-1}$ can play the role of $t_{q-1}$ and
satisfies the claim.

Let $U=S \cup \{t_1, \dots, t_{q-1}\}$.  We prove by induction on $k$
that if $U$ has size $N$, then there is an induced path of size
$N^{\frac{1}{k-1}}$ in the subgraph of $G$ induced by $U$.  If $k=2$
then the subgraph induced by $U$ is a path.  Now assume that $k\ge 3$.
We consider the number of vertices of $U$ intersected by $t_1, \dots,
t_{q-1}$.  Suppose that one of these vertices intersects more than
$N^{\frac{k-2}{k-1}}$ vertices.  In the graph induced by these
vertices (and the corresponding vertices of $S$), the maximal clique
has size at most $k-1$, so by the induction hypothesis it contains an
induced path of size $N^{\frac{1}{k-1}}$.  Now suppose that each of
$t_1, \dots, t_{q-1}$ intersects at most $N^{\frac{k-2}{k-1}}$
vertices of $U$.  Then we can build a path, starting with $s_q$, such
that after each vertex $s_i$ with $i\neq 1$ we take the vertex
$t_{i-1}$, and after each vertex $t_{i}$, we take the smallest vertex
of $S$ adjacent to $t_i$.  Thus we obtain an induced path with at
least $N^{\frac{1}{k-1}}$ vertices.

Since an interval graph with $n$ vertices and of maximum clique size
$k$ is properly $k$-colorable, it contains a stable set of size at
least $\tfrac{n}{k}$.  It follows that $N \geq q \geq \frac{n}{k}$,
and therefore $G$ contains an induced path of size
$(\frac{n}{k})^{\frac{1}{k-1}}$, as desired.
\end{proof}

It follows from the preceding three lemmas that any interval graph of
maximum clique size $k$ containing a path on $n$ vertices also
contains an induced path of size %
$$\left( \frac{
(\log_{k+2}(n)-\log_{k+2}((k+2)!))^{\tfrac1{k-1}}}{k}\right)^{\tfrac1{k-1}}\le
c_k (\log n)^{\tfrac1{(k-1)^2}},$$ for some contant $c_k$.  This
proves Theorem~\ref{thm:int}.

This result shows that interval graphs satisfy
Conjecture~\ref{conj:1}, but unfortunately we do not have a
construction showing that our lower bound has the correct order of
magnitude in the specific case of interval graphs.  It might still be
the case that interval graphs with long paths and bounded clique
number have induced paths of polynomial size.  Improving
Lemma~\ref{lem:int1} might be the key in proving such a result (since
the other two lemmas gives polynomial bounds).

%%%%%%%
% \clearpage
\section{Conclusion}

We proved that $k$-trees with long paths have induced paths of
logarithmic size.  However, this does not give any clue whether
\emph{partial} $k$-trees with paths of size $n$ have induced paths of
size polylogarithmic in $n$.  We only proved that one cannot hope to
obtained a bound exceeding $\Omega((\log n)^{\frac{1}{2k}}))$.

We believe that proving Conjecture~\ref{conj:1} for partial $k$-trees
will also imply (with a reasonable amount of work, based on our result
for graphs embedded on fixed surfaces) that Conjecture~\ref{conj:1}
holds for any proper minor-closed class, and any proper class closed
under topological minor (using the corresponding structure theorems).

\end{document}